\documentclass[11pt, a4paper, oneside]{article}
\usepackage{amsmath, amssymb, amsthm, amsfonts, amsxtra, latexsym, amscd,
pb-diagram,  graphics, hyperref, setspace}
\usepackage [all]{xy}
\theoremstyle{plain}
\newtheorem*{dn}{\bf Definition}

\newcommand{\tx}{\otimes }

\newcommand{\ri}{\rightarrow }

\newcommand{\Fm}{\widetilde{F}}

\DeclareMathOperator{\Coker}{Coker} 
\DeclareMathOperator{\Hom}{Hom} 
\DeclareMathOperator{\Aut}{Aut} \DeclareMathOperator{\Ext}{Ext}
\DeclareMathOperator{\Ker}{Ker}\DeclareMathOperator{\Obs}{Obs}
\DeclareMathOperator{\Red}{Red}

\newtheorem{thm}{\bf Theorem}
\newtheorem{lem}[thm]{\bf Lemma} % Lemma
\newtheorem{pro}[thm]{\bf Proposition} % Proposition
\newtheorem{hq}[thm]{\bf Corollary} % Corollary

\theoremstyle{definition}
\begin{document}

\centerline{\Large\bf Group Extensions of the Co-type of a Crossed
Module}\vspace{0.2cm} \centerline{\Large\bf and Strict Categorical
Groups}

\vspace{0.5cm} \centerline{\textsc{Nguyen Tien Quang}}

\centerline{\textit{ Department of Mathematics,   Hanoi National
University of Education,}}
  %136 Xuanthuy Street,
  % Hanoi,
  %    Vietnam }
 \centerline{\textit{cn.nguyenquang@gmail.com}}
%\maketitle \setcounter{tocdepth}{1}

\begin{abstract}
{\it Prolongations} of a group extension can be studied in a more
general situation that we call group extensions  of the co-type of a
crossed module. Cohomology classification of such extensions is
obtained by applying the obstruction theory of monoidal functors.
\end{abstract}


\section{Introduction}
A description of group extensions by means of factor sets leads to a
  close relationship between the extension problem of a type of algebras
   and the corresponding cohomology theory. This allows to study extension
    problems using cohomology as an effective method
\cite{Hoch-Serre53}.

Let $A$ and $\Pi$ be two groups, $A$ abelian. An extension of $A$ by
$\Pi$ is a short exact sequence
\begin{equation}\label{mrn}0\ri
A\stackrel{i}{\ri}B\stackrel{p}{\ri}\Pi\ri 1.\end{equation}
 A classical theorem in homological algebra asserts that the group of
isomorphism classes of extensions of $A$ by $\Pi$ with a fixed
operator  $\varphi:\Pi\ri$ Aut$A$ is isomorphic to the second
cohomology group $H^2_\varphi(\Pi,A)$, \cite{MacL63}.  After that,
the group $H^2$ was applied to the problem of classifying all group
extensions in the  different situations. This theorem has been made
more precise by establishing a categorical equivalence between the
category of extensions and a certain category whose objects are
2-cocycles \cite{Laven-Ro}. With the notion of  a {\it categorical
group} (or a {\it Gr-category} \cite{Sinh78}),  many aspects of
group extension problem are  raised to a categorical level which
help to obtain applications in algebra (see
 \cite{C2002}, \cite{Vitale03}). This article belongs
to this type.

The article is derived from the following classical problem. For a
group extension (\ref{mrn}) and a group homomorphism $\eta:\Pi'\ri
\Pi$, it follows from the existence of the pull-back of the pair
$(\eta,p)$  that there is an extension $E\eta$ making the following
diagram commute
%\begin{equation}
%\label{ct0}
\[\begin{diagram}
\xymatrix{ E\eta:&0\ar[r]&A\ar[r]^{j'} \ar@{=}[d]
&B'\ar[r]^{p'}\ar[d]^{\beta}&\Pi'\ar[r]
\ar[d]^\eta&1\\
E:&0\ar[r]&A\ar[r]^{j}&B\ar[r]^p&\Pi\ar[r]&1 }
\end{diagram}\]
%\end{equation}
(see \cite{MacL63} - Chapter III, \cite{Hilton70} - Chapter IV).

The problem is that with a given extension $E'$ and a homomorphism
  $\eta:\Pi'\ri \Pi$, let us find all extensions $E$ of $A$ by
$\Pi$ such that $E'=E\eta$. Then, the extension $E$ is said to be a
{\it $\eta$-prolongation} of $E'$. A brief and general description
of this problem was introduced in \cite{Weiss} (Proposition 5.1.1).
In \cite{QPC} we show the  better descriptions in the case of
%Mét tr­êng hîp riªng, mµ ta sÏ tr×nh bµy d­íi ®©y øng víi
  the central  extensions and
$\eta$ is an injection. Each prolongation induces a model which is
``dual" to a group extension of the type of a crossed module (see
Section 6). This leads to the notion of group {\it extension of
co-type} of a crossed module studied in this paper.

The plan of this paper is, briefly, as follows. In Section  2  we
recall reduced categorical groups, monoidal functors of  type
$(\varphi, f)$. In Section 3 we show the relation between the
category of crossed modules and the category of strict
Gr-categories, which is a useful tool in the next proofs. Next, we
introduce the notion of a $\zeta$-extension of the co-type of a crossed
module in Section 4, and we  construct the obstruction theory of a
$\zeta$-extension (Theorem \ref{dlc1}). In Section 5 we present
Schreier theory for $\zeta$-extensions of the co-type of a crossed
module (Theorem \ref{pl}).  The last section is devoted to applying
the results of previous sections to the problem of prolongations of
a group extension in \cite{QPC}.

%%%%%%%%%%%%%%%%%%%%%%%%%%%%%%%%%%%%%

\section{Preliminaries}
For later use, we recall here  some basic facts and results about
 categorical groups  (see   \cite{QTC}, \cite{Sinh78}).

\indent A {\it  categorical group }  is a monoidal category
$(\mathbb G ,\tx, I, \mathbf{a}, \mathbf{l}, \mathbf{r})$ in which
every object is invertible and the underlying category is a
groupoid. If $(F, \Fm, F_\ast)$ is a monoidal functor between
categorical groups, the isomorphism $ F_\ast:I'\ri FI$ can be
deduced from  $F$ and $\Fm$. Thus, we will refer to  $(F, \Fm)$ as a
monoidal functor.

Two monoidal functors $(F, \Fm)$ and $(F', \Fm')$ from $\mathbb G$
to $\mathbb G'$ are {\it homotopic} if there is a \emph{natural
monoidal equivalence} (or a {\it homotopy})
$\alpha:(F,\Fm,F_\ast)\ri (F',\Fm',F_\ast')$, which is a natural
equivalence such that
%\begin{equation*}
$F'_\ast=\alpha_I\circ F_\ast.$ %\end{equation*}

 Each categorical group $\mathbb G$ determines three invariants, as follows:

 1. The set $\pi_0\mathbb G$ of isomorphism classes of the objects in
$\mathbb G$ is a group where the operation is induced by the tensor
product in $\mathbb G.$

2. The set $\pi_1\mathbb G$ of automorphisms of the unit object $I$
is a $\pi_0\mathbb G$-module.

3. An element $[k]\in H^3(\pi_0\mathbb G,\pi_1\mathbb G)$ is induced
by the associativity constraint of $\mathbb G$.
%Moreover, $\pi_1\mathbb G$ is a $\pi_0\mathbb G$-module with the
%action
%\begin{equation}
%$$\label{ct0a}su=\gamma_{X}^{-1}\delta_{X}(u),\ X\in s,\ s\in \pi_0\mathbb G,\ u\in \pi_1\mathbb
%G,$$
%\end{equation}
%where $\gamma_X,\delta_X$ are defined by the following commutative
%diagrams
% $$\begin{CD}
%X @>\gamma_X(u) >> X \\
%@A {\bf l}_X AA     @AA {\bf l}_X A \\
%I \tx X @>u \tx id >> I \tx X
%\end{CD}
%\qquad \qquad\qquad
%\begin{CD}
%X @>\delta_X(u) >> X \\
%@A {\bf r}_X AA     @AA {\bf r}_X A \\
%X \tx I @>id \tx u >> X \tx I.
%\end{CD}$$
%\end{thm}
% The {\it reduced}  categorical group $S_{\mathbb G}$ of a

Based on the data: a group $\Pi$, a $\Pi$-module $A$ and $k\in
Z^3(\Pi,A)$, we construct a categorical group, denoted by
Red$(\Pi,A,k)$ whose objects   are elements $x\in \Pi$ and the
morphisms are automorphisms $(x,a):x\ri x,$ where $x \in \Pi,a\in
A$. The composition of two morphisms is induced by the addition in
 $A$
$$(x,a)\circ (x,b)=(x,a+b).$$
%The category $S_{\mathbb G}$  is equivalent to $\mathbb G$ by
%canonical equivalences constructed as follows. For each $s=[X] \in
%\pi_0\mathbb G,$ choose a representative $X_s\in \mathbb G$ such
%hat $X_1=I$, and for each $X \in s,$ choose an isomorphism $i_X:\
%X_s\rightarrow X$ such that $i_{X_s}=id_{X_s}$. The family
% $(X_s, i_X)$ is a {\it stick} of the category  $\mathbb G$ if
%$$
%\ i_{I\otimes X_s}={\bf l}_{X_s},\ i_{ X_s\otimes I}={\bf
%r}_{X_s}.$$
% For a stick $(X_s, i_X)$, we obtain two functors
%\[\begin{cases}
%G:\mathbb G\ri  S_{\mathbb G}\\
%G(X)=[ X]=s\\
%G(X \stackrel {f}{\ri}Y)=(s,\gamma_{X_s}^{-1}(i_{Y}^{-1}fi_X))
%\end{cases}\qquad\qquad
%\begin{cases}
%H: S_{\mathbb G}\ri  \mathbb G\\
%H(s)=X_s\\
%H(s,u)=\gamma_{X_s}(u).
%\end{cases}\]
%\indent Two functors  $G$ and  $H$ are are categorical equivalences
%by natural transformations
%$$\alpha=(i_{X}): HG\cong id_{\mathbb G},\quad\ \beta=id:GH\cong id_{ S_{\mathbb G}}.$$
%They are called \emph{canonical} equivalences.

%Via the structure transportation by the quadruple $(G, H, \alpha,
%\beta),$ $S_{\mathbb G}$  becomes a  categorical group together with
%the following operations
The tensor products are given by
\begin{gather*}
x\tx y=x.y,\quad x,y\in\Pi,\\
(x,a)\tx(y,b)=(xy,a+xb),\quad a,b\in A.
\end{gather*}
The unit constraints of the  categorical group $\Red(\Pi,A,k)$ are
 strict, and its associativity constraint
  is ${\bf a}_{x,y,z}=(xyz,k(x,y,z))$.

  In the case where $\Pi,A,[k]$ are three invariants of a
  categorical group $\mathbb G$ then $\Red(\Pi,A,k)$ is monoidally
  equivalent to $\mathbb G$ and it is called a {\it reduction} of
  $\mathbb G$, hence denoted by $\mathbb G(k)$.

 %Further, the equivalences
%$G,\ H$ together with natural isomorphisms
%\begin{equation}\label{ct1a}\widetilde{G}_{X,Y}=G(i_X\tx i_Y),\
% \widetilde{H}_{s,t}=i_{X_s\tx X_t}^{-1}.
% \end{equation}
%are monoidal equivalences.

%The  categorical group $S_{\mathbb G}$  is  a \emph{reduction} of
%the  categorical group $\mathbb G.$ $S_{\mathbb G}$  is said to be
%of  {\it type} $(\Pi, A, k)$, or of {\it type} $(\Pi,A)$ if
%$\pi_0\mathbb G, \pi_1\mathbb G$ are replaced by a group $\Pi$ and a
%$\Pi$-module $A$, respectively.

%Let $\mathbb S=(\Pi, A,k),\ \mathbb S'=(\Pi',A',k')$ be  categorical
%groups.
A functor $F: \Red(\Pi,A,k)\rightarrow \Red(\Pi',A',k')$ is  of {\it
type} $(\varphi,f)$ if
\begin{equation*}
F(x)=\varphi(x), \ F(x,a)=(\varphi(x), f(a)),
\end{equation*}
where $\varphi:\Pi\ri \Pi'$, $f:A\ri A'$ are group homomorphisms
satisfying $f(xa)=\varphi (x)f(a)$, for $x\in \Pi, a\in A.$ Note
that if $\Pi'$-module $A'$ is considered as a $\Pi$-module under the
action  $xa'=\varphi(x).a'$, then $f:A\ri A'$ is a homomorphism of
$\Pi$-modules. In this case, we call $(\varphi,f)$ a {\it pair of
homomorphisms} and call
\begin{equation}\label{ctr} \xi=\varphi^\ast
k'-f_\ast k\in Z^3(\Pi,A')
\end{equation}
an {\it obstruction} of the functor $F$, where $\varphi^\ast,f_\ast$ are
canonical homomorphisms
$$ Z^3(\Pi,A)\stackrel{f_\ast}{\longrightarrow} Z^3(\Pi,A')
\stackrel{\varphi^\ast}{\longleftarrow}Z^3(\Pi',A').$$
% If $F:\mathbb S\rightarrow
%\mathbb S'$ is a functor of  type $(\varphi, f)$, then the function
% \begin{equation}\label{ct1b}
 %\xi=\varphi^{\ast}k'-f_{\ast}k\in Z^3(\Pi,A')
 %\end{equation}
 %$$k=\varphi^{\ast}\xi'-f_{\ast}\xi$$
%is called an {\it obstruction} of the functor $F.$ Thus, we state
%the following result.
The results on monoidal functors of type $(\varphi,f)$ stated in
\cite{QTC} are summarized in the following proposition.

\begin{pro}\label{tk}  Let $\mathbb G$ and $\mathbb G'$ be two
 categorial groups, $\mathbb G(k)$ and $\mathbb
G'(k')$ be their reductions, respectively.

 $\mathrm{i)}$ Every monoidal functor $(F,\widetilde{F}):\mathbb G\ri\mathbb G'$
 induces  one $\mathbb G(k)\ri \mathbb G'(k')$ of type $(\varphi,f)$.

 $\mathrm{ii)}$  Every  monoidal functor $\mathbb G(k)\ri \mathbb G'(k')$
  is a functor of type $(\varphi,f)$.

$\mathrm{iii)}$   A   functor $F:\mathbb G(k)\ri \mathbb G'(k')$ of
type $(\varphi, f)$ is realizable, that is, it induces a
 monoidal functor, if and only if its obstruction
$[\xi]$ vanishes in
   $H^3_\Gamma(\Pi, A')$. Then, there is a bijection
 \begin{equation*}
     {\mathrm{Hom}}_{(\varphi, f)}[\mathbb G(k),\mathbb G'(k')]\leftrightarrow H^2_\Gamma(\Pi,
 A'),
\end{equation*}
where $\Hom_{(\varphi, f)}[\mathbb G(k),\mathbb G'(k')]$ is the set
of all homotopy classes of monoidal  functors of type $(\varphi, f)$
from $\mathbb G(k)$ to $\mathbb G'(k')$.
\end{pro}

%\newpage
\section{ Categorical groups associated to a crossed module}

A categorical group is {\it strict}, according to Joyal and Street
\cite{J-S}, if all of its constraints are strict and every object
has a strict inverse ($x\otimes y=1=y\otimes x$). Brown and Spencer
\cite{Br76} called it a {\it $\mathcal G$-groupoid}. The authors of
\cite{Br76} showed that there is a categorical equivalence between
the category of crossed modules and that of $\mathcal G$-groupoids,
  and hence crossed modules can be studied by means of
category theory. The Brown-Spencer equivalence has recently
developed for the category of (braided) crossed bimodules (see
\cite{Q14}, Theorems 4.3, 4.4).

\begin{dn} \emph{ A {\it crossed module} is a quadruple $(B,D,d,\theta)$ where
$d:B\ri D,\;\theta:D\ri$ Aut$B$ are group homomorphisms such that
the following relations hold\\
\indent $C_1.$ $\theta d=\mu$,\\
\indent $C_2.$  $ \label{ct4b}d(\theta_x(b))=\mu_x(d(b)),\;x\in D,
b\in B,$\\
 where $\mu_x$ is an inner automorphism given by
conjugation of $x$.} \end{dn}

%%%%%%%%%%%%%%%%%%%%%%%%%%%%
% For example, each onto homomorphism $p:B\ri D$,  $\Ker p\subset Z(B)$, is a crossed module
%  where the action of
% $D$ on $B$ is the conjugation.

\noindent\begin{dn} \emph{ A {\it homomorphism}
$(f_1,f_0):(B,D,d,\theta)\ri (B',D',d',\theta')$ of crossed modules
consists of group homomorphisms   $f_1:B\ri B'$, $f_0:D\ri D'$ satisfying\\ %the following diagram commutes
%\begin{equation}\label{gr1}\begin{diagram}
%\node{B}\arrow{e,t}{d}\arrow{s,l}{f_1}\node{D}\arrow{s,r}{f_0}\\
%\node{B'}\arrow{e,b}{d'}\node{D'}
%\end{diagram}\end{equation}
%%and $f_1$ is an {\it operator homomorphism}, that is,
%\begin{equation}\label{gr2}
%f_1(\theta_xb)=\theta'_{f_0(x)}f_1(b),
%\end{equation}
\indent $H_1. \ f_0d=d'f_1,$\\
\indent $H_2. \ f_1(\theta_xb)=\theta'_{f_0(x)}f_1(b)$,\\
 for all $x\in D,b\in B$.}\end{dn}
In the present paper, the crossed module $(B,D,d,\theta)$ is
sometimes denoted by
 $B\stackrel{d}{\rightarrow}D$.
For convenience, we denote by the addition for the operation in $B$
and by the multiplication for that in $D$.

The following properties follow from the definition of a crossed
module.
\begin{pro} \label{md2}
Let $(B,D,d,\theta)$ be a crossed module.

$\mathrm{i)}\; \mathrm{Ker}d\subset Z(B)$.

$\mathrm{ii)} \;\mathrm{Im}d$ is a normal subgroup in $D$.

$\mathrm{iii)}$ The homomorphism $\theta$ induces a homomorphism
$\varphi:D\ri \mathrm{Aut}(\mathrm{Ker}d)$ given by
$$\varphi_x=\theta_x|_{\mathrm{Ker}d}.$$

$\mathrm{iv)}\; \mathrm{Ker}d$ is a left $\mathrm{Coker}d$-module
with the action
$$sa=\varphi_x(a),\ \ a\in\mathrm{ Ker}d, \ x\in s\in \mathrm{Coker}d.$$
\end{pro}
%%%%%%%%%%%%%%%%%%%%%%%%%%%%%%%%%%%%%%%%%%%
As mentioned above, a categorical group can be seen as a crossed
module \cite{Br76}, \cite{J-S}. To help motivate the reader, we
present this fact in detail.

 For each crossed module $(B,D,d, \theta)$, one can
construct a strict  categorical group  $\mathbb G_{B\rightarrow
D}=\mathbb G$, called the  categorical group {\it associated to} the
crossed module $B\ri D$, as follows.
$$\mathrm{Ob}\mathbb G=D, \ \mathrm{Hom}(x,y)=\{b\in B\ |\ x=d(b)y\},$$ where $x,y$ are objects of $\mathbb G$.
 The composition of two
morphisms is given by
$$(x\stackrel{b}{\ri}y\stackrel{c}{\ri}z)=(x\stackrel{b+c}{\ri}z).$$
The tensor functor is given by  $x\otimes y=xy$ and
%vµ víi hai mòi tªn $(x\stackrel{b}{\ri}y), (x'\stackrel{b'}{\ri}y')$ th×
\begin{equation}\label{mt}
(x\stackrel{b}{\ri}y)\otimes(x'\stackrel{b'}{\ri}y')=(xx'\stackrel{b+\theta_yb'}{\longrightarrow}yy').
\end{equation}
%By the definition of a crossed module, we can easily check that
% $\mathbb P$ is a   categorical group with
%the identity constraints.

Conversely, for a strict  categorical group $(\mathbb G, \otimes)$,
we define a  crossed module $C_\mathbb G=(B,D,d,\theta)$ as follows.
Set
 \begin{equation} D=\mathrm{ Ob}\mathbb G,\;\ B=\{x\xrightarrow{b}1 | x\in D  \}.\notag\end{equation}
%\begin{equation} \notag \end{equation}
The operations on $D$ and on $B$ are given by
%\begin{align*}
$$xy=x\otimes y,\;\ b+c=b\otimes c, $$%\end{align*}
respectively. Then, the set $D$ becomes a group in which the unit is
$1$, the inverse of $x$ is $x^{-1}$ ($x\otimes x^{-1}=1$). The set
$B$ is a group in which the unit is the morphism
$(1\xrightarrow{id_1} 1)$ and the inverse of $(x\xrightarrow{b} 1)$
is the morphism $(x^{-1}\xrightarrow{\overline{b}} 1 )(b\otimes
\overline{b}=id_1)$.

The homomorphisms $d:B\rightarrow D$ and $\theta: D\rightarrow \Aut
B$ are respectively given by
%\begin{align*}
$$d(x\xrightarrow{b}  1)=x,$$% \end{align*}
$$\theta _y(x\xrightarrow{b}1)= (yxy^{-1}\xrightarrow{id_y + b + id_{y^{-1}}}  1).$$ %\end{align*}
%respectively. So defined, it is easy to verify that  $(B,D,d,\theta
%)$ is a crossed module.
%%%%%%%%%%%%%%%%%%%%%%%%%%%%%%%
%%%%%%%%%%%%%%%%%%%%%%%%%%%%%%%%

%In the next lemmas, denote by $\mathbb P=\mathbb P_{B\ri D}$ and
%$\mathbb P'=\mathbb P_{B'\ri D'}$ two  categorical groups associated
%to crossed modules $(B,D,d,\theta)$ and $(B',D',d',\theta')$,
%respectively.

The following result shows the relationship between homomorphisms of
crossed modules and monoidal functors of associated categorical
groups.

\begin{pro}[\cite{CQT2}\label{t1}]
Let $(f_1,f_0):(B,D,d,\theta)\ri(B',D',d',\theta')$ be a homomorphism of
crossed module.

 $\mathrm{i)}$ There is a functor  $F:\mathbb G_{B\ri
D}\ri\mathbb G_{B'\ri D'}$ given by
$$F(x)=f_0(x),\;F(b)=f_1(b),$$
 where $x\in \mathrm{Ob}\mathbb G$, $b\in \mathrm{Mor}\mathbb G,$

  $\mathrm{ii)}$ Natural isomorphisms $\widetilde{F}_{x,y}$
together with $F$ is a monoidal functor if and only if
 $\widetilde{F}_{x,y}=\varphi(\overline{x},\overline{y})$, where $\varphi\in Z^2(\mathrm{Coker}d,\mathrm{Ker}d').$
\end{pro}

{\it Note.} In the category of $\mathcal G$-groupoids in
\cite{Br76}, the morphisms $(F,\widetilde{F})$ satisfy
$\widetilde{F}=id.$

\section {Group extensions of the co-type of a crossed module}
In this section we introduce a concept which is ``dual" to the
concept of group extension of type $B\stackrel{d}{\rightarrow}D$ in
\cite{Br94,Br96}. As will be showed later, it is also regarded as a
generalization of the prolongation problem of group extensions
\cite{QPC}.

\begin{dn} \emph{Let $d: B\rightarrow D$ be a crossed module. A group
\emph{extension} of
 $A$ of \emph{ co-type }
$B\stackrel{d}{\rightarrow}D$  is  a diagram of group homomorphisms
\[  \begin{diagram}
\xymatrix{\;&\;&  B \ar[r]^d \ar[d]_\beta &D \ar@{=}[d]&\; \\
 0 \ar[r]&A \ar[r]^j &E \ar[r]^p &D  \ar[r] &1,}
\end{diagram}\]
where the bottom  row is exact,  $j(A)\subset Z(E)$, the pair
$(\beta,id_D)$ is a morphism of crossed modules.}\end{dn}
%Ta nãi r»ng $E$ lµ mét {\it më réng t©m kiÓu m«®un chÐo} $B\ri D.$

%This notion can be seen as a ``dual" version of the notion of an
%extension of  $B$ by $Q$ of type
%$B\stackrel{d}{\rightarrow}D$ which is referred in  \cite{Br94,Br96}.

Since the bottom row is exact and since  $p\circ \beta\circ i=d\circ
i=0$, where $i:\Ker d\ri B$ is an inclusion, there exists a unique
homomorphism $\zeta: \mathrm{Ker}d\ri A$ such that the left hand
side square commutes
\begin{equation}\label{ct0} \begin{diagram}
\xymatrix{\;&\mathrm{Ker}d \ar[r]^i\ar@{.>}[d]_\zeta& B \ar[r]^d \ar[d]_\beta &D \ar@{=}[d]&\; \\
 0 \ar[r]&A \ar[r]^j&E \ar[r]^p &D  \ar[r] &1.}
\end{diagram}
\end{equation}

This homomorphism is defined by \begin{equation}\label{ld}j(\zeta
c)=\beta(i c),\;c\in \mathrm{Ker}d.
\end{equation}
Moreover,
$\zeta$ depends only on the equivalence class of the extension $E$.

{\it Note on terminologies.} Since the homomorphism $\theta'$ of the
crossed module $E\stackrel{p}{\rightarrow}D$ is the conjugation and
since $j(A)\subset Z(E)$, $\theta'_x$ acts on $A$ as an identity.
Thus, the group $A$ can be seen as a $D$-module with the trivial
action. Then,
\begin{equation}\label{ld1}\zeta(sc)=\zeta(c),\;\ s\in \Coker d, c\in \Ker d.
\end{equation}
%víi $x\in D$, $c\in$ Ker$d$.
Indeed, By Proposition \ref{md2}, $\theta_x(c)\in \Ker d,$ so one
has
\[\begin{aligned}j\zeta(\theta_x(c))&\stackrel{(\ref{ld})}{=}\beta i(\theta_x
c)=\beta(\theta_x c)\\
&\stackrel{(H_2)}{=}\theta'_x(\beta
c)=\theta'_x(j\zeta(c))=j\zeta(c).
\end{aligned}\]
Since $j$ is injective, we obtain (\ref{ld1}). Thus, it defines a
trivial $\Coker d$-module structure on $\mathrm{Im}\zeta$.

 The homomorphism $\zeta:\Ker d\ri A$ satisfying the condition \eqref{ld1} is called an {\it abstract $\zeta$-kernel}
 of the crossed module $B\stackrel{d}{\rightarrow}D$.
An extension of $A$  of co-type $B\stackrel{d}{\rightarrow}D$
inducing $\zeta:\Ker d\ri A$ is said to be an  {\it extension of the
abstract $\zeta$-kernel}, or a {\it $\zeta$-extension} of co-type
$B\stackrel{d}{\rightarrow}D$.

$\bullet$ {\it The obstruction theory: the case $\zeta $ is
surjective}

 From now on, assume that $\zeta: \Ker d\ri A$ is an
onto homomorphism. We use the obstruction theory of monoidal
functors to deal with the existence of  $\zeta$-extensions.

 Let $\mathbb G= \mathbb G_{B\rightarrow
D}$ be the   categorical group associated to crossed module
$B\rightarrow D$. Since $\pi_0\mathbb G=\Coker d$ and $\pi_1\mathbb
G=\Ker
d$, the reduced  categorical group $\mathbb G(k)$ is of form %Khi ®ã, Ann-ph¹m trï thu gän
$$\mathbb G(k)= \Red(\mathrm{Coker} d, \mathrm{Ker} d,k), \; [k]\in H^3 (\mathrm{Coker} d, \mathrm{Ker} d),$$
where the associativity constraint $k$
 is defined as follows. Choose a set of representatives $\left\{x_s\ |\ s\in \mathrm{Coker}d\right\}$ in $D$.
For each $x\in s$ choose an element $b_x\in B$ satisfying
$x_s=d(b_x)x$, $b_{x_s}=0$. According to \cite{Sinh78}, the family
$(x_s, b_x)$ is called a {\it stick}. It defines a monoidal functor
 $(H,\widetilde{H}): \mathbb G(k)\ri \mathbb G$ by
 $$H(s)=x_s,\ H(s,a)=a,\ \widetilde{H}_{r,s}=-b_{x_rx_s}.$$
Then, $k$ is determined by the following commutative diagram
 \begin{equation}\label{qs}
\begin{CD}
x_s(x_r x_t )@>x_s\tx\widetilde{H}_{r,t}>> x_sx_{rt}@> \widetilde{H}_{s,rt}>>x_{srt} \\
@|                       @.                        @VVk(s,r,t)V \\
(x_s x_r) x_t @>\widetilde{H}_{s,r}\tx x_t >> x_{sr} x_{t} @>
\widetilde{H}_{sr,t}>> x_{rst}.
\end{CD}
%\tag{13}
\end{equation}
By the relation \eqref{mt}, this diagram implies
$$\theta_{x_s}(\widetilde{H}_{r,t})+\widetilde{H}_{s,rt}+k(s,r,t)=\widetilde{H}_{s,r}+\widetilde{H}_{sr,t}.$$
We write $k=\delta (\widetilde{H})$ even though the function
$\widetilde{H}$ takes values in  $B$. The cohomology class
$$\Obs(\zeta)=[\zeta_\ast k]\in H^3 (\mathrm{Coker} d, A)$$ is
called  {\it the obstruction} of the abstract $\zeta$-kernel.

The onto homomorphism $\zeta:\Ker d\ri A$ induces a quotient
category $\mathbb G/\Ker\zeta$ with the same objects of $\mathbb G
 (=D),$ but morphisms are homotopy classes of morphisms in $\mathbb
G$, i.e.,  elements of the group $\overline{B}=B/\Ker\zeta$. The
category $\mathbb G/\Ker\zeta$ is just the categorical group
associated to the crossed module  $(\overline{B}, \overline{d},
\overline{\theta}, D)$ induced by the crossed module $(B, d, \theta,
D)$.

\begin{lem}\label{lema1}  If the obstruction  $\Obs(\zeta)$ vanishes in $H^3(\Coker d, A)$,
there exists a monoidal fuctor $\Red(D,A,0)\ri \mathbb G/\Ker\zeta$.
 \end{lem}
\begin{proof} If $\mathrm{Obs}(\zeta)$ vanishes in $H^3(\mathrm{Coker}d, A)$, then
 $\zeta_\ast k=\delta g$, where  $g:(\mathrm{Coker}d)^2\ri A$. Consider a functor
$$F: \Red(D,A,0)\ri \Red(\Coker d, A, \delta g),$$
for $F=(q,id)$, where $q$ is the natural projection. The obstruction
of $F$ is
$$q^\ast(\delta g)=\delta(q^\ast g).$$
Thus, $F$ together with $\widetilde{F}=q^\ast g$ is a monoidal
functor. It follows the existence of a monoidal functor from
$\Red(D,A,0)$ to $\mathbb G/\Ker\zeta$.
\end{proof}

\begin{lem}\label{lema2} Each monoidal functor $\Red(D,A,0) \ri \mathbb G/\Ker\zeta$ defines a $\zeta$-extension of co-type
 $B\stackrel{d}{\rightarrow}D$.
\end{lem}
\begin{proof}
%$\bullet$ 
Construction of the crossed product from a monoidal
functor $(\Gamma, \widetilde{\Gamma}): \Red(D,A,0) \ri \mathbb
G/\Ker\zeta$.

The morphism $\widetilde{\Gamma}$ defines an associated function
$g:D^2\ri A$ by $\widetilde{\Gamma}_{s,r}=(1,g(s,r))$. Now, we set
  $\varphi: \Coker d \rightarrow \Aut \overline{B}$ by
\begin{align} \varphi_s(\overline{b})= \overline{\theta} _{x_s}\overline{b}(=\overline{\theta _{x_s}(b)}).\label{eq7'}
 \end{align}

Since $x_rx_s=\overline{d}(\widetilde{\Gamma}_{r,s})x_{rs}$, the
functions $\varphi, g$ satisfy the rule
\begin{align*} \varphi_s \varphi_r = \mu_{g(s,r)}\varphi_{sr}. \label{eq8'} \end{align*}
Since $\delta g=0$, according to Lemma 8.1 \cite {MacL63} one can
defines a crossed product $E_g= [\overline{B},g,\varphi,\Coker d]$.
Namely, $E_g= \overline{B}\times \Coker d$ and the operation on
$E_g$ is
\begin{equation}\label{tic}
(\overline{b},s)+(\overline{c},r)=
(\overline{b}+\varphi_s(\overline{c})+g(s,r), sr). \end{equation} In
this group $(0,1)$ is the zero, while the negative of the element
 $(\overline{b},s)$ is $ (\overline{b'},s^{-1})$, where $
\varphi_s(\overline{b'}) =-\overline{b}-g(s,s^{-1}).$
One obtains an exact sequence
\begin{align*} \mathcal E_g:0\rightarrow \ A \xrightarrow{j_g} E_g \xrightarrow{p_g}D \rightarrow 1,   \end{align*}
where
%\begin{align*}
$ j_g(\zeta(c))=(\overline{c},1),\; p_g(\overline{b},s) = db.x_s.$ %\end{align*}
Indeed,
$$p_g j_g(\zeta(c))=p_g(\overline{c},1)=dc.x_s=1,$$
and for $(\overline{b},s)\in$ Ker$(p_g)$, then
$p_g(\overline{b},s)=db.x_s=1$. By the uniqueness of the
representation in  $D$, we have $db=1$ and $x_s=1$, it follows that
$b\in$ Ker$d$ and $s=1$, or $(\overline{b},s)\in$ Im$(j_g)$.

%$\bullet$  
We prove that  $j_g(A)\subset Z(E_g)$. For $b,c\in B$,
one has
\begin{equation}\label{ht}
\mu_{(\overline{b},s)}(\overline{c},1)=({\mu_{\overline b}\varphi_s(\overline c)},1)
\end{equation}
If $c\in \Ker d $, then by \eqref{ld1}, ${\varphi_s(\overline
c)}=\overline{c}$. Hence,
$$\mu_{(\overline{b},s)}(\overline{c},1)=(\mu_{\overline
b}(\overline c),1)=(\overline {b+c-b},1)=(\overline {c},1).$$

%$\bullet$
 Since $j_g(A)\subset Z(E_g)$ and $p_g$ is a surjection,
$E_g\stackrel{p_g}{\rightarrow}D$ is a crossed module in which the
homomorphism $\theta':D\ri \Aut E_g$ is the conjugation. To define
the morphism $(\beta,id_D)$ of crossed modules, one set
  $$\beta: B\ri E_g,\;\ \beta(b)=(\overline{b},1).$$
This correspondence is a homomorphism thanks to the relation
\eqref{tic}.
%%%%%%%%%%%%%%%%%%%%%%%%%%
Clearly, $p_g\circ \beta=d$. Moreover, for all $c\in B$ and
$x=db.x_s\in D$, we have
$$\beta(\theta_x(c))=\beta(\theta_{db}(\theta_{x_s}(c)))=(\overline{\mu_b\theta_{x_s}(c)},1)\stackrel{\eqref{eq7'}}{=}(\mu_{\overline{b}}\varphi_s(\overline{c}),1).$$
Since $\theta'_x=\mu_{(\overline{b},s)}$,
\begin{align*}
\theta'_x\beta(c)&=\mu_{(\overline{b},s)}(\overline{c},1)\stackrel{\eqref{ht}}{=}(\mu_{\overline{b}}\varphi_s(\overline{c}),1)
\end{align*}
%nªn $(\mu_b\thea_{x_s}(c),1)$.
%MÆt kh¸c ta cã
Thus, the relation $H_2$ holds, and $\mathcal E_g$ is a
$\zeta$-extension of co-type $B\stackrel{d}{\rightarrow}D$.
\end{proof}
We state one of the paper's main results.
\begin{thm}\label{dlc1} Let  $\zeta: \mathrm{Ker} d\rightarrow A$ be the abstract
$\zeta$-kernel of the crossed module $B\stackrel{d}{\rightarrow}D$.
Then, the vanishing of the obstruction  $\Obs(\zeta)$
   in $H^3(\Coker d,A)$ is  necessary and sufficient for there to exist a $\zeta$-extension
   of co-type $B\stackrel{d}{\rightarrow}D$.

\end{thm}
\begin{proof}
{\it Necessary condition}. Let $\mathcal E$ be a $\zeta$-extension
of co-type $B\ri D$ satisfying the diagram \eqref{ct0}. Then, the
reduced categorical group of the categorical group $\mathbb G'$
associated to the crossed module
 $E\stackrel{p}{\rightarrow} D$ is $\Red(1,A,0)$.  By Proposition \ref{t1}, the pair $(\beta,id_D)$
 determines a monoidal functor $(F, \widetilde{F}): \mathbb G\ri \mathbb G'$.
 By Proposition \ref{tk}, $(F, \widetilde{F})$ induces a  monoidal functor of type $(0,\zeta)$ from
  $\Red(\mathrm{Coker}d, \mathrm{Ker}d,k)$ to $\Red(1,A,0)$.
Also by Proposition \ref{tk}, the obstruction $[\zeta_\ast k]$ of
the pair $(0,\zeta)$ vanishes in $H^3(\Coker d, A)$.

{\it Sufficient condition}. It follows directly from Lemma
\ref{lema1} and Lemma \ref{lema2}.
\end {proof}

\section {Classification theorem}
\begin{dn} \emph{ Two $\zeta$-extensions of co-type
 $(B,D,d,\theta)$,
\begin{align*}0\rightarrow A\xrightarrow{j}E\xrightarrow{p} D\rightarrow 1,  \;\ \  B\stackrel{\beta}{\rightarrow}E  \end{align*}
\vspace{-30pt}
\begin{align*}0\rightarrow A\xrightarrow{j'}E'\xrightarrow{p'} D\rightarrow 1, \;\ \ B\stackrel{\beta'}{\rightarrow}E'   \end{align*}
are {\it equivalent} if there is an isomorphism $\omega:E\ri E'$
such that $\omega$$j=j'$,  $p'\omega=p$ and
$\omega\beta=\beta'$.}\end{dn}

We denote by
   $$\mathrm{Ext}_{B\ri D}(D,A,\zeta)$$
   the set of all equivalence classes of $\zeta$-extensions of co-type
 $B\ri D$ inducing $\zeta$.
We describe this set by means of the set
$$\mathrm{Hom}_{(0,\zeta)}[\Red(D,A,0),{\mathbb G}/\Ker\zeta]$$
of homotopy classes of monoidal functors of type $(0,\zeta)$ from
$\Red(D,A,0)$ to ${\mathbb G}/\Ker\zeta$.
First,let $q:B\ri\overline{B}=B/\Ker\zeta$, and $\sigma:D\ri \Coker
d$ be the natural projections, one states the following lemma.
\begin{lem}\label{bd01}
If $\zeta$ is surjective, then the commutative diagram
\eqref{ct0} induces a short exact sequence
 \begin{align}0\rightarrow \ \overline{B} \xrightarrow{\varepsilon} E\xrightarrow{\sigma p} \Coker d
  \rightarrow 1, \label{eq01'}   \end{align}
where $\varepsilon(b+\Ker  \zeta )=\beta(b) $.
\end{lem}
\begin{proof}
Obviously, $\sigma p$ is  surjective. It is easy to see that
$\Ker\beta=\Ker\zeta$, so $\varepsilon$ is injective. The diagram
\eqref{ct0} implies $\sigma p \varepsilon(\overline{b})=\sigma p
\beta({b})=\sigma d(b)=1$, this means that the above sequence is
semi-exact. For $e\in$ Ker$(\sigma p)$, $p(e)\in$ Ker$\sigma=$
Im$d$, and hence $p(e)=d(b)=p\beta(b)=p\varepsilon(\overline{b})$.
Then,
 $e= \varepsilon (\overline{b})+ja$. Since $ja=j\zeta(c)=\beta(c)=\varepsilon(\overline{c})$,
$e=\varepsilon(\overline{b+c})\in  \mathrm{Im}\varepsilon$. Thus,
the sequence (\ref{eq01'}) is exact.
\end{proof}

 \begin{lem}\label{14'}
Each $\zeta$-extension of co-type  $B\ri D$ is equivalent to a
crossed product extension which is constructed from a monoidal
functor of type $(0,\zeta)$,
$(\Gamma,\widetilde{\Gamma}):\Red(D,A,0)\ri {\mathbb G}/\Ker\zeta$.
\end{lem}

\begin{proof}
Let $E$ be a $\zeta$-extension of co-type
$B\stackrel{d}{\rightarrow}D$.
 By the proof of Theorem \ref{dlc1},
there is a monoidal functor
$(\Gamma,\widetilde{\Gamma}):\Red(D,A,0)\ri {\mathbb G}/\Ker\zeta$.
  By Lemma \ref{lema2},
  the crossed product
 $E_g$, where $g$ is the function associated with $\widetilde{\Gamma}$, is a
$\zeta$-extension of co-type $B\stackrel{d}{\rightarrow}D$ in which
$$\beta_g:B\ri E_g,\ b\mapsto (\overline{b},1).$$

Thanks to the exact sequence \eqref{eq01'} in Lemma \ref{bd01}, each
element of $E$ can be represented uniquely as
$\varepsilon{\overline{b}}+e_s$,
 where $\left\{e_s, s\in \Coker d\right\}$
is a set of representatives of $\Coker d$ in $E$. It is easy to
check that the correspondence
$$\omega:E\ri E_g,\;\ \varepsilon \overline{b}+e_s\mapsto (\overline{b},s)$$
is a group isomorphism. Moreover, $\omega$ makes two extensions $\mathcal
E$ and $\mathcal E_g$  equivalent.
\end{proof}

\begin{thm}[Schreier theory \label{pl} for extensions of co-type of a crossed module]

%Cho $B\ri D$ lµ mét module chÐo vµ ®ång cÊu $\zeta:Kerd\ri A$ tháa m·n ®iÒu kiÖn (5).
If $\zeta$-extensions of co-type $B\stackrel{d}{\rightarrow}D$ exist,
then there is a bijection
$$\Omega: \mathrm{Ext}_{B\ri D}(D,A,\zeta)\rightarrow\mathrm{Hom}_{(0,\zeta)}[\Red(D,A,0),\mathbb G/ \Ker\zeta].$$
\end{thm}
\begin{proof} The correspondence
$E \mapsto (\Gamma,\widetilde{\Gamma}) $ in Lemma \ref{14'} defines
a correspondence $[E]\mapsto [(\Gamma,\widetilde{\Gamma})]$.  The
fact that $\Omega$ is injective
 implies  by following steps.

{\it Step 1:  If  monoidal functors $(\Gamma,\widetilde{\Gamma})$
and $(\Gamma',\widetilde{\Gamma}')$ are homotopic, then two
extensions $\mathcal E_g$ and $\mathcal E_{g'}$ are equivalent.}

Let $\Gamma,\Gamma':\Red(D,A,0)\ri{\mathbb G}/\Ker\zeta$ be two
monoidal functors and
 $\alpha:\Gamma\ri \Gamma'$ be a homotopic. Then, the following
 diagram commutes
\[\begin{diagram}
\node{\Gamma s\Gamma r}\arrow{e,t}{\widetilde{\Gamma}}\arrow{s,l}{\alpha_s\otimes
\alpha_r}\node{\Gamma sr}\arrow{s,r}{\alpha_{sr}}\\
\node{\Gamma's\Gamma'r}\arrow{e,b}{\widetilde{\Gamma'}}\node{\Gamma'sr.}
\end{diagram}\]
Since the morphisms  $\alpha_s$ are of forms
  $ (1,a_s)$,  it follows from the
 above diagram that
\begin{equation}\label{Sch}g(s,r)-g'(s,r)=a_s+a_r-a_{sr}=(\delta a)(s,r).
\end{equation}
Since $\zeta$ is  surjective, $a_s=\zeta(z_s)$, where $z:\Coker
d\ri \Ker d$ is a normalized function.

 Then, by \eqref{Sch},
$\alpha$ determines a map $\omega:E_g\ri E_{g'}$ by
\begin{align}\label{htbs}
(\overline{b},s)&\mapsto[\overline{b+z_s},s].\end{align} By the
relation \eqref{ld1} and by the definition of operations in
$E_g,E_{g'}$, the map $\omega$ is a group homomorphism. Further, it
makes two extensions $\mathcal E$ and $\mathcal E_g$  equivalent.

{\it Step  2: If two extensions $\mathcal E_g$ and $\mathcal E_{g'}$
are equivalent, then $(\Gamma,\widetilde{\Gamma})$ and
$(\Gamma,\widetilde{\Gamma})$ are homotopic.}

Let $\mathcal E_g$ and $\mathcal E_{g'}$ be equivalent via the
isomorphism
 $\omega:E_g\ri E_{g'}$. From $p_g=p_{g'}\omega:E_g\ri E_g'\ri D,$ %tÝnh giao ho¸n cña h×nh vu«ng
%thø hai trong biÓu ®å (\ref{mrtd})
it follows that $\omega$ is of the form (\ref{htbs}), where
$z:\Coker d\ri \Ker d$ is a normalized function. Since  $\omega$ is
a homomorphism, $\alpha=\zeta_\ast k$ is a homotopy between $\Gamma$
and $\Gamma'$.

It follows from Lemma \ref{lema2} that $\Omega$ is surjective.
\end{proof}

It follows from Proposition \ref{tk} and Theorem \ref{pl} that
\begin{hq}\label{hq01}
If $\zeta$-extensions of co-type $B\stackrel{d}{\rightarrow}D$ exist,
then there is a bijection
$$\Ext_{B\ri D}(D,A,\zeta)\leftrightarrow H^{2}(\Coker d, A).$$
\end{hq}

\section{Prolongations of a group extension}
In this section we show an application of $\zeta$-extensions of co-type
of a crossed module in order to obtain the results on prolongations
of a group extension in the sense of  \cite{QPC}. Given a
commutative diagram of group homomorphisms
\begin{equation}\label{gr9}\begin{diagram}
\xymatrix{\mathcal  B:&0\ar[r]&\Ker\pi\ar[r]^{i}\ar[d]^{ \zeta
}&B\ar[r]^{\pi}\ar[d]^{\beta}&\Pi\ar[r]
\ar[d]^\eta&1\\
\mathcal E:&0\ar[r]&A\ar[r]^{j}&E\ar[r]^p&D\ar[r]&1 }
\end{diagram}
\end{equation}
where the rows are exact, $\Ker\pi\subset ZB$, $\eta$ is a normal
monomorphism (in the sense that  $\eta \Pi$ is a normal subgroup of
$D$) and $ \zeta $ is an epimorphism.  Then, $\mathcal E$ is said to
be a $( \zeta ,\eta)$-\emph{prolongation} of  $\mathcal B$.

%In this section, we denote
% $\Pi_0=\Coker\eta$, and let $\sigma:D\ri \Pi_0$ be  the natural projection.
%  Obviously, $\Ker \beta =\iota(\Ker \zeta) $.

%For convenience, we write the  operations in   $\Pi, D, \Pi_0$ as
%multiplication and in other groups as addition, even though the
%groups $B, E$ are non-necessarily abelian.

For the quotient group  $\overline{B}=B/\Ker\zeta$, the
homomorphisms $i, \eta\pi, \zeta, \beta$
 in the commutative diagram \eqref {gr9} induce the homomorphisms $\iota,
  d, \overline{\zeta}, \overline{\beta}$, respectively, such that the following diagram
commutes

%\begin{Ex} \emph{ If $(B,D,d,\theta)$ is a crossed module, then by Proposition
%\ref{md2}, $\mathrm{Im}d$ is a normal subgroup of $D$. Therefore,
%each $\zeta$-extension of type $B\stackrel{d}{\rightarrow}D$ can be
%seen as a $(\zeta,\eta)$-prolongation of the extension
%$$0\rightarrow \Ker d \rightarrow B\rightarrow \mathrm{Im}d\rightarrow 1,$$
%where the inclusion map $\eta:\mathrm{Im}d\rightarrow D$.}
%\end{Ex}

% Now, let $\mathcal E$ be a
 %$( \zeta ,\eta)$-prolongation of $\mathcal B$.

 % By Lemma \eqref{bd01},
 %the commutative diagram \eqref {gr9} induces a homomorphism
 %$\varepsilon:\overline{B} \rightarrow E$ by $\varepsilon(b+\Ker  \zeta )=\beta(b) $.
%
 %
%\begin{align}0\rightarrow \ \overline{B} \xrightarrow{\varepsilon} E\xrightarrow{\sigma p} \Pi_0 \rightarrow 1. \label{eq1'}
%  \end{align}
%Let $\pi_0:\overline{B}\ri \Pi$ be a homomorphism induced by $\pi$,
%and  $\overline{\zeta}:  \Ker\pi/\Ker\zeta\ri A$ the canonical
%isomorphism  $\overline{c}\mapsto \zeta(c)$, the following diagram
%commutes
\begin{equation}\label{ct0'}
\begin{diagram}
\xymatrix{\;&\mathrm{Ker}d \ar[r]^\iota\ar@{.>}[d]_{\overline{\zeta}}& \overline{B} \ar[r]^d \ar[d]^{\overline{_\beta}} &D \ar@{=}[d]&\; \\
 0 \ar[r]&A \ar[r]^j&E \ar[r]^p &D  \ar[r] &1,}
\end{diagram}
\end{equation}
%where $d=\eta\pi_0:\overline{B}\ri \Pi\ri D$, $\Ker
%d=\Ker(\pi_0)=\Ker\pi/\Ker\zeta$.
Besides, according to Theorem 2
\cite{QPC}, $\mathcal E$  induces a homomorphism $\theta :
D\rightarrow \Aut \overline{B}$ such that the quadruple
$(\overline{B},D,d, \theta) $ is a crossed module.

\begin{thm}
 $\mathcal E$ is a $\overline{\zeta}$-extension of co-type $(\overline{B}, D, d, \theta)$.
% trong ®ã $\overline{\zeta}:\Ker\pi/\Ker\zeta\ri A$ lµ ®¼ng cÊu c¶m sinh bëi $\zeta$.
\end{thm}
\begin{proof}
In the diagram \eqref{ct0'}, since the bottom row is exact  and
$jA\subset ZE$ (Theorem 10 \cite{QPC}), the epimorphism $p:E\ri D$
together with the conjugation in $E$ is a crossed module. It is easy
to see that the pair  $(\overline{_\beta}, id_D)$ is a homomorphism of
crossed modules, so $\mathcal E$ is a $\overline{\zeta}$-extension
of co-type $(\overline{B}, D, d, \theta)$.
\end{proof}

$\bullet$ {\it The problem of prolongations of a group extension.}
%\section{Obstructions to prolongations of extensions}

Given a diagram of group homomorphisms
\begin{align*}   \begin{diagram}
\xymatrix{ \mathcal E:&0\ar[r]&\Ker\pi\ar[r]^{i}\ar[d]^{ \zeta }&B\ar[r]^{\pi}&\Pi\ar[r]\ar[d]^\eta&1\\
&&A&&D }
\end{diagram}
\end{align*}
where the row is exact,  $i$ is an inclusion map, $\Ker\pi\subset
ZB$,
 $\eta$ is a normal monomorphism, $ \zeta $ is surjective,  %tháa m·n $ \alpha (xa_0)=x \alpha (a_0)$  víi $x\in \Coker \gamma =\Pi_0,$
and a group homomorphism $\theta : D\rightarrow \Aut(\overline{B})$
such that the quadruple $(\overline{B}, D,d, \theta)$ is a crossed
module (where the notations $\overline{B},d$ are defined as above).
These data are denoted by the triple $(\zeta,\eta,\theta)$, called a
{\it pre-prolongation} of $\mathcal E$. A
 $(\zeta,\eta)$-prolongation of  $\mathcal E$ inducing  $\theta$ is also called a \emph{covering} of the pre-prolongation
$(\zeta,\eta,\theta)$.

The ``prolongation problem'' is that of finding whether there is any
covering of  the pre-prolongation $(\zeta,\eta,\theta)$ of $\mathcal
E$ and, if so, how many.

According to \cite{QPC}, each pre-prolongation $(\zeta,\eta,\theta)$
of $\mathcal E$ induces an obstruction $k$. This obstruction is just
the obstruction of an abstract $\zeta$-kernel of the crossed module
$\overline{B}\stackrel{d}{\rightarrow}D.$ Thus, from the results on
crossed modules in previous sections, one obtains the solution of
the problem of prolongations of a group extension (Theorem 8 and
Theorem 15 in \cite{QPC}).

\begin{thm}\label{ad}  Let $(\zeta, \eta, \theta)$ be a pre-prolongation.

 $\mathrm{i)}$ The vanishing of the obstruction
$[\overline{\zeta}_\ast k]$ in $H^3 (\Coker d, A)$ is necessary and
sufficient for there to exist a
 covering of $(\zeta, \eta,
\theta)$.

$\mathrm{ii)}$  If $[\overline{\zeta}_\ast k]$ vanishes, there is a
bijection
$$\mathrm{Ext}_{(\zeta,\eta)}(D, A)\leftrightarrow H^2(\Coker d, A),$$
where $\Ext_{(\zeta,\eta)}(D,A)$ is the set of equivalence classes
of  $(\zeta,\eta)$-prolongations of the extension   $\mathcal B$
inducing $\theta$.
\end{thm}
\begin{proof}
i) According to Theorem \ref{dlc1}, the vanishing of
$[\overline{\zeta}_\ast k]$ in $H^3 (\Coker d, A)$ is necessary and
sufficient for there to exist a  $\overline{\zeta}$-extension
$\mathcal E$ of co-type $\overline{B}\stackrel{d}{\rightarrow}D$.
Thanks to the following diagram, this is equivalent to the fact that
 $\mathcal E$ is a covering of the pre-prolongation
$(\zeta,\eta, \theta)$,
\[
  \begin{diagram}
\xymatrix{ \mathcal B:\;\;\;\;\;\;\;0 \ar[r]& \mathrm{Ker}\pi \ar[r]^i\ar[d]^{\zeta_0} &B  \ar[r]^\pi \ar[d]^{p_0}  & \Pi \ar[r]\ar[d]^\eta & 1 \\
 \;& \mathrm{Ker} d \ar[r]^{\iota} \ar[d]^{\overline{\zeta}}  &\overline{B}  \ar[r]^{d} \ar[d]^{\overline{^\beta}}  & D \ar@{=}[d] & \; \\
 \mathcal E:\;\;\;\;\;\;\;0 \ar[r]& A \ar[r]^{j}  &E  \ar[r]^{p}   & D \ar[r]& 1.}
\end{diagram}\]
ii) It is clear that two coverings of the pre-prolongation
$(\zeta,\eta, \theta)$ are equivalent if and only if they are two
equivalent $\overline{\zeta}$-extensions of co-type
$\overline{B}\stackrel{d}{\rightarrow}D$, that is, there is a
bijection
$$\mathrm{Ext}_{(\zeta,\eta)}(D,A)\leftrightarrow \mathrm{Ext}_{\overline{B}\ri D}(D,A,\overline{\zeta}).$$
Now, by Corollary  \ref{hq01}, we have the bijection

\vspace{0.2cm}
$\ \quad \qquad \qquad \qquad \mathrm{Ext}_{(\zeta,\eta)}(D,A)\leftrightarrow H^{2}(\Coker d,A).$
\end{proof}

\begin{center}
{}
\end{center}

\end{document}